\documentclass[twoside,leqno,10pt, A4]{amsart}
\usepackage{amsfonts}
\usepackage{amsmath}
\usepackage{amscd}
\usepackage{amssymb}
\usepackage{amsthm}
\usepackage{amsrefs}
\usepackage{latexsym}
\usepackage{mathrsfs}
\usepackage{bbm}
\usepackage{enumerate}
\usepackage{graphicx}

\usepackage{amsfonts}
\usepackage{amsmath}
\usepackage{amscd}
\usepackage{amssymb}
\usepackage{amsthm}
\usepackage{amsrefs}
\usepackage{latexsym}
\usepackage{mathrsfs}
\usepackage{bbm}
\usepackage{amscd}
\usepackage{amssymb}
\usepackage{amsthm}
\usepackage{amsrefs}
\usepackage{latexsym}
\usepackage{mathrsfs}
\usepackage{bbm}
\usepackage{enumerate}
\usepackage{graphicx}
\usepackage{color}
\begin{document}

\newtheorem{theorem}[subsection]{Theorem}
\newtheorem{proposition}[subsection]{Proposition}
\newtheorem{lemma}[subsection]{Lemma}
\newtheorem{corollary}[subsection]{Corollary}
\newtheorem{conjecture}[subsection]{Conjecture}
\newtheorem{prop}[subsection]{Proposition}
\numberwithin{equation}{section}
\newcommand{\mr}{\ensuremath{\mathbb R}}
\newcommand{\mc}{\ensuremath{\mathbb C}}
\newcommand{\dif}{\mathrm{d}}
\newcommand{\intz}{\mathbb{Z}}
\newcommand{\ratq}{\mathbb{Q}}
\newcommand{\natn}{\mathbb{N}}
\newcommand{\comc}{\mathbb{C}}
\newcommand{\rear}{\mathbb{R}}
\newcommand{\prip}{\mathbb{P}}
\newcommand{\uph}{\mathbb{H}}
\newcommand{\fief}{\mathbb{F}}
\newcommand{\majorarc}{\mathfrak{M}}
\newcommand{\minorarc}{\mathfrak{m}}
\newcommand{\sings}{\mathfrak{S}}
\newcommand{\fA}{\ensuremath{\mathfrak A}}
\newcommand{\mn}{\ensuremath{\mathbb N}}
\newcommand{\mq}{\ensuremath{\mathbb Q}}
\newcommand{\half}{\tfrac{1}{2}}
\newcommand{\f}{f\times \chi}
\newcommand{\summ}{\mathop{{\sum}^{\star}}}
\newcommand{\chiq}{\chi \bmod q}
\newcommand{\chidb}{\chi \bmod db}
\newcommand{\chid}{\chi \bmod d}
\newcommand{\sym}{\text{sym}^2}
\newcommand{\hhalf}{\tfrac{1}{2}}
\newcommand{\sumstar}{\sideset{}{^*}\sum}
\newcommand{\sumprime}{\sideset{}{'}\sum}
\newcommand{\sumprimeprime}{\sideset{}{''}\sum}
\newcommand{\sumflat}{\sideset{}{^\flat}\sum}
\newcommand{\shortmod}{\ensuremath{\negthickspace \negthickspace \negthickspace \pmod}}
\newcommand{\V}{V\left(\frac{nm}{q^2}\right)}
\newcommand{\sumi}{\mathop{{\sum}^{\dagger}}}
\newcommand{\mz}{\ensuremath{\mathbb Z}}
\newcommand{\leg}[2]{\left(\frac{#1}{#2}\right)}
\newcommand{\muK}{\mu_{\omega}}
\newcommand{\thalf}{\tfrac12}
\newcommand{\lp}{\left(}
\newcommand{\rp}{\right)}
\newcommand{\Lam}{\Lambda_{[i]}}
\newcommand{\lam}{\lambda}
\def\L{\fracwithdelims}
\def\om{\omega}
\def\pbar{\overline{\psi}}
\def\phis{\phi^*}
\def\lam{\lambda}
\def\lbar{\overline{\lambda}}
\newcommand\Sum{\Cal S}
\def\Lam{\Lambda}
\newcommand{\sumtt}{\underset{(d,2)=1}{{\sum}^*}}
\newcommand{\sumt}{\underset{(d,2)=1}{\sum \nolimits^{*}} \widetilde w\left( \frac dX \right) }

\newcommand{\hf}{\tfrac{1}{2}}
\newcommand{\af}{\mathfrak{a}}
\newcommand{\Wf}{\mathcal{W}}

\newtheorem{mylemma}{Lemma}
\newcommand{\intR}{\int_{-\infty}^{\infty}}

\theoremstyle{plain}
\newtheorem{conj}{Conjecture}
\newtheorem{remark}[subsection]{Remark}

\makeatletter
\def\widebreve{\mathpalette\wide@breve}
\def\wide@breve#1#2{\sbox\z@{$#1#2$}%
     \mathop{\vbox{\m@th\ialign{##\crcr
\kern0.08em\brevefill#1{0.8\wd\z@}\crcr\noalign{\nointerlineskip}%
                    $\hss#1#2\hss$\crcr}}}\limits}
\def\brevefill#1#2{$\m@th\sbox\tw@{$#1($}%
  \hss\resizebox{#2}{\wd\tw@}{\rotatebox[origin=c]{90}{\upshape(}}\hss$}
\makeatletter

\title[Shifted moments of modular $L$-functions to a fixed level]{Shifted moments of modular $L$-functions to a fixed level}

\author[P. Gao]{Peng Gao}
\address{School of Mathematical Sciences, Beihang University, Beijing 100191, China}
\email{penggao@buaa.edu.cn}

\author[L. Zhao]{Liangyi Zhao}
\address{School of Mathematics and Statistics, University of New South Wales, Sydney NSW 2052, Australia}
\email{l.zhao@unsw.edu.au}

\begin{abstract}
 We establish upper bounds for shifted moments of modular $L$-functions to a fixed prime level under the generalized Riemann hypothesis.  
\end{abstract}

\maketitle



\section{Introduction}
\label{sec 1}

 The density conjecture of N. Katz and P. Sarnak \cite{K&S} asserts that each reasonable family of $L$-functions is associated to a classical compact group (unitary, symplectic, and orthogonal).  This is regarded as the family's symmetry type can be ascertained by computing the corresponding $n$-level density of lower-lying zeros. In this matter, it was shown \cite{O&S} that the symmetric type of the family of quadratic Dirichlet $L$-functions is symplectic, while we know from \cite{HuRu} that of the family of Dirichlet $L$-functions to a fixed modulus is unitary. \newline
 
In addition to the above method of determining the symmetry type of a family of $L$-functions, it was observed in \cite{G&Zhao2022} that one may also obtain the information by evaluating (shifted) moments of the family.  For instance, B. Szab\'o \cite{Szab} proved under the generalized Riemann hypothesis (GRH) that for a large fixed modulus $q$, any positive integer $k$, real tuples ${\bf a} =  (a_1, \ldots, a_k), {\bf t} =  (t_1, \ldots, t_k)$  such that $a_j \geq 0$ and $|t_j| \leq q^A$ for a fixed positive real number $A$,
\begin{align}
\begin{split}
\label{eqn:shiftedMoments}
 \sum_{\chi\in X_q^*}\big| L\big( \tfrac12+it_1,\chi \big) \big|^{a_1} \cdots \big| L\big( \tfrac12+it_{k},\chi \big) \big|^{a_{k}} \ll_{\bf{t}, \bf{a}} &  \varphi(q)(\log q)^{(a_1^2+\cdots +a_{k}^2)/4} \prod_{1\leq j<l\leq k}  \big|\zeta(1+i(t_j-t_l)+\tfrac 1{\log q}) \big|^{a_ja_l/2}, 
\end{split}
\end{align}
   where $X_q^*$ denotes the set of primitive Dirichlet characters modulo $q$, $\varphi$ the Euler totient function and $\zeta(s)$ the Riemann zeta function. 
  Assuming $|t_j|\leq  X^A$ for a large real number $X$, the authors \cite{G&Zhao2024-3} proved, under GRH, that
\begin{align}
\label{eqn:shiftedMomentsquad}
\begin{split}
 \sumstar_{\substack{(d,2)=1 \\ d \leq X}} & \big| L\big( \tfrac12+it_1,\chi^{(8d)} \big) \big|^{a_1} \cdots \big| L\big(\tfrac12+it_{k},\chi^{(8d)}  \big) \big|^{a_{k}} \\
& \ll X(\log X)^{(a_1^2+\cdots +a_{k}^2)/4} \prod_{1\leq j<l \leq k} \Big|\zeta \Big( 1+i(t_j-t_l)+\tfrac 1{\log X} \Big) \Big|^{a_ia_j/2} \Big|\zeta \Big(1+i(t_j+t_l)+\tfrac 1{\log X} \Big) \Big|^{a_ia_j/2} \\
& \hspace*{2cm} \times \prod_{1\leq j\leq k} \Big|\zeta \Big(1+2it_j+\tfrac 1{\log X} \Big) \Big|^{a^2_i/4+a_i/2},
\end{split}
\end{align}
 where $\sum^*$ denotes a sum over positive and square-free integers, $\chi^{(8d)}=\leg {8d}{\cdot}$ is the Kronecker symbol. \newline
 
The bounds in both \eqref{eqn:shiftedMoments} and \eqref{eqn:shiftedMomentsquad} are believed to be sharp. In fact, the authors \cite{G&Zhao24-11} proved that this is the case for $k=2$ and primes $q$ under GRH. Here one sees (upon setting $q=X$ in \eqref{eqn:shiftedMoments}) that these (sharp) bounds are indeed distinct as the symmetry types of the corresponding family of $L$-functions differ. \newline
 
 As \eqref{eqn:shiftedMoments} and \eqref{eqn:shiftedMomentsquad} exhibit bounds for unitary and symplectic families, curiosity naturally arises about the remaining case of orthogonal type.  For this, we turn our attention to a typical orthogonal family, studied already by H. Iwaniec, W. Luo and P. Sarnak in their seminal work \cite{ILS}, by computing its one-level density there. To describe the family, let $H^*_\kappa(N)$ denote the set of all holomorphic cusp forms of even weight $\kappa$ that are newforms of level $N$.  For each $f \in H^*_\kappa(N)$, its Fourier expansion at infinity is written as
\[
f(z) = \sum_{n=1}^{\infty} \lambda_f (n) n^{(\kappa -1)/2} e(nz), \quad \mbox{where} \quad e(z) = \exp (2 \pi i z).
\]
  The modular $L$-function $L(s, f)$ associated to $f$ is defined for $\Re(s) > 1$ to be
\begin{align*}
L(s, f ) &= \sum_{n=1}^{\infty} \frac{\lambda_f(n)}{n^s}. 
\end{align*}
  
  Recall (see \cite[Section 2]{ILS}) that the function $L(s, f )$ is entire and satisfies the functional equation given by
\begin{align}
\label{equ:FElevelN}
\Lambda (s, f) = \left(\frac{\sqrt{N}}{2\pi} \right)^s \Gamma (s + \tfrac{\kappa -1}{2}) L(s, f)
= \epsilon_f \Lambda (1- s, f),
\end{align}
 where $\epsilon_f=i^\kappa\mu(N)\lambda_f(N)N^{1/2} \in \{\pm 1\}$. Here $\mu$ is the M\"obius function. \newline
  
 It follows that $H^*_\kappa(N)$ can be splited into two disjoint subsets, $H^{\pm}_\kappa(N) :=\{f \in H^*_\kappa(N) :\epsilon_f =\pm 1\}$.  From \cite{ILS}, we learn that the symmetry type of the families of $L$-functions $L(s,f)$ for $f \in H^*_\kappa(N)$ is orthogonal, $SO(\text{even})$ for $f \in H^{+}_\kappa(N)$ and $SO(\text{odd})$ for $f \in H^{-}_\kappa(N)$. \newline
 
Out aim of this paper is to further demonstrate the symmetry type of the family of $L$-functions $L(s,f)$ for $f \in H^*_\kappa(N)$ by estimating from above the shifted moments of this family. Our result is as follows.
\begin{theorem}
\label{t1}
 With the notation as above and the truth of GRH, let $k\geq 1$ be a fixed integer and $a_1,\ldots, a_{k}$, $A$ fixed positive real numbers. Suppose that $N$ is a large prime number and $t=(t_1,\ldots ,t_{k})$ a real $k$-tuple with $|t_j|\leq  N^A$. Then
\begin{align}
\label{cqLbounds}
\begin{split}
  \sum_{\substack{f \in H^*_k(N)}} \prod^k_{j=1} \big| L\big( \tfrac12 +it_j,f \big) \big|^{a_j}  \ll &  N. 
\end{split}
\end{align}
  The implied constant depends on $k$, $A$ and the $a_j$'s, but not on $N$ or the $t_j$'s.
\end{theorem}

   We note that the bounds given in \eqref{cqLbounds} differ from those given in \eqref{eqn:shiftedMoments} and \eqref{eqn:shiftedMomentsquad}, which reflects the orthogonal symmetry nature of the family of the underlying $L$-functions. We also note that by a straightforward modification of the arguments used in the proof of Theorem \ref{t1} and by making use of \cite[(2.17)]{HM07}, we may establish upper bounds for the shifted moments of the families of $L$-functions $L(s,f)$ for $f \in H^{\pm}_\kappa(N)$ as well. \newline

Finally, we end this section by contrasting our result with those in \cite{IwSa00}.  Theorem~\ref{t1} estimates from above unweighted moments in \eqref{cqLbounds} while those in \cite{IwSa00} use harmonic weights.  These weights have non-trivial size (see \cite[(2.9)]{IwSa00}) and lead to certain logarithmic factors, e.g. \cite[(2.21)]{IwSa00}.  As we obtain an unweighted bound, these logarithmic factors are expectedly absent here.

\section{Preliminaries}
\label{sec 2}

   In this section, we gather several necessary auxiliary results.
\subsection{Cusp form $L$-functions}
\label{sec:cusp form}

    We reserve the letter $p$ for a prime number throughout in this paper.  For any $f \in H^*_{\kappa} (N)$, the associated modular $L$-function $L(s, f)$ for $\Re(s) > 1$ has an Euler product (see \cite[(3.2)]{ILS})
\begin{align*}
L(s, f ) &
 = \prod_{p } \left(1 - \frac{\lambda_f (p) }{p^s}  + \frac{\chi_0(p)}{p^{2s}}\right)^{-1}=\prod_{p} \left(1 - \frac{\alpha_p }{p^s} \right)^{-1}\left(1 - \frac{\beta_p }{p^s} \right)^{-1},
\end{align*}
  where $\chi_0$ denotes the principal character modulo $N$. 
 By Deligne's proof \cite{D} of the Weil conjecture, we know that for $p \nmid N$, 
\begin{align*}
|\alpha_{p}|=|\beta_{p}|=1, \quad \alpha_{p}\beta_{p}=1.
\end{align*}
  Moreover, we infer from \eqref{equ:FElevelN} that $|\lambda_f(N)|=1/N$ and that $\lambda_f(p) \in \mr$ for $p \nmid N$. It follows that  we have $\lambda_f (1) =1$ and when $N$ is a prime, we have $|\lambda_f(n)| \leq d(n)$, for $n \geq 1$, where $d(n)$ is the number of positive divisors $n$. \newline

   From \cite[(2.73)]{ILS}, we see that for $N>1$,
\begin{align}
\label{Hpmsize}
 |H^{\pm}_{\kappa} (N)|= \frac {k-1}{24} N+O((\kappa N)^{5/6}).
\end{align}
  Here and after, we write $|S|$ for the cardinality of a set $S$. \newline

  We quote the following result from \cite[Lemma 2.8]{HM07} on the multiplicative properties of $\lambda_f$.
\begin{lemma}
\label{lemma:cubicclass}
 If $f \in H^*_{\kappa} (N)$, then 
\begin{align*}
 \lambda_f(m)\lambda_f(n)=\sum_{\substack{d \mid (m,n) \\ (d,N)=1}}\lambda_f \Big(\frac {mn}{d^2}\Big).
\end{align*} 
   In particular, if $(m,n) = 1$, then
\begin{align*}
 \lambda_f(m)\lambda_f(n)=\lambda_f(mn).
\end{align*} 
Moreover, for a prime $p \nmid N$,  
\begin{align*}
\begin{split}
 \lambda_f(p)^{2m}=& \sum^{m}_{r=0}\left (\binom {2m}{m-r}-\binom {2m}{m-r-1}  \right ) \lambda_f(p^{2r}), \quad \mbox{and} \\
  \lambda_f(p)^{2m+1}=& \sum^{m}_{r=0}\left (\binom {2m+1}{m-r}-\binom {2m+1}{m-r-1}  \right ) \lambda_f(p^{2r+1}).
\end{split}
\end{align*} 
\end{lemma} 
 
   We note from the above lemma $\lambda_f(p)^{2m}$ (resp. $\lambda_f(p)^{2m+1}$) can be expresses as a sum over only even (resp. odd) powers $\lambda_f(p)^{2r}$ (resp. $\lambda_f(p)^{2r+1}$). \newline
  
  For integers $m,n$, we define the Kloosterman sum $S(m,n, q)$ by 
\begin{align*}
  S(m,n, q)=\sumstar_{\substack{ u \shortmod q }} e \Big(\frac {mu+n\overline u}{q}\Big),
\end{align*}
   where $u\overline{u} \equiv 1 \pmod q$.  Moreover, we define
\begin{align*}
 \Delta^*_{\kappa, N}(n)=\sum_{f \in H^*_{\kappa} (N)}\lambda_f(n).
\end{align*}   
  
   We quote from \cite[Lemma 2.9]{HM07} on $\Delta^*_{\kappa, N}(n)$, which follows from Propositions 2.1, 2.11, 2.15 and Lemma 2.12 of \cite{ILS}.
\begin{lemma}
\label{lemma:delta}
  Let $X,Y$ be two real numbers with $X<N$. For $N$ prime and $(n,N^2)\mid N$, we write
\begin{align*}
 \Delta^*_{\kappa, N}(n)=\Delta^{'}_{\kappa, N}(n)+\Delta^{\infty}_{\kappa, N}(n), 
\end{align*} 
   where
\begin{align*}
 \Delta^{'}_{\kappa, N}(n)=\frac {(\kappa-1)N}{12\sqrt{n}}\delta_{n, \square_Y}+\frac {(\kappa-1)N}{12}\sum_{\substack{(m,N)=1 \\ m \leq Y}}\frac {2\pi i^{\kappa}}{m}
 \sum_{\substack{c \equiv 0 \shortmod N \\ c \geq N}}\frac {S(m^2,n;c)}{c}J_{\kappa-1} \Big(4\pi \frac {\sqrt{m^2n}}{c}\Big),
\end{align*} 
$J_{\kappa-1}$ is the $J$-Bessel function, $\delta_{n, \square_Y}=1$ if $n=m^2$ with $m \leq Y$ and zero otherwise.  Let $(a_m)$ be a sequence of complex numbers satisfying 
\begin{align*}
 \sum_{\substack{(m,nN)=1 \\ m < M}}\lambda_f(m)a_m \ll (n\kappa N)^{\varepsilon}, \quad \log M  \ll \log (\kappa N)
\end{align*} 
   for every $f \in H^*_{\kappa}(1) \cup H^*_{\kappa}(N)$ with the implied constant depending on $\varepsilon$ only. Then if $(n,N^2)|N$, we have
\begin{align*}
 \sum_{\substack{(m,nN)=1 \\ m < M}}\Delta^{\infty}_{\kappa, N}(nm)a_m \ll \frac {\kappa N}{\sqrt{(n, N)}}\Big (\frac 1X+\frac 1{\sqrt{Y}}\Big ) (n\kappa NXY)^{\varepsilon}.
\end{align*}    
\end{lemma}    
   
   We next quote \cite[Lemma 2.10]{HM07} for a simplification of the above lemma. 
\begin{lemma}
\label{lemma:deltasimplified}
   If $N$ is prime and $(n,N)=1$, then for any $\varepsilon>0$, 
\begin{align*}
\begin{split}
 \frac 1{|H^*_{\kappa}(N)|}\Delta^{'}_{\kappa, N}(n)=& \frac {1}{\sqrt{n}}\delta_{n, \square_Y}+O(n^{(\kappa-1)/2}N^{-\kappa+1/2+\varepsilon}), \quad \mbox{and} \\
  \frac 1{|H^*_{\kappa}(N)|}\Delta^{'}_{\kappa, N}(Nn) \ll &  n^{1/2}N^{-3/2+\varepsilon}.
\end{split}
\end{align*}
\end{lemma}       
\subsection{Sums over primes}
\label{sec2.4}

We have the following standard bounds for sums over rational prime numbers.
\begin{lemma}
\label{RS} Let $x \geq 2$ and $\alpha \geq 0$. We have, for some constant $b$,
\begin{align*}
\sum_{p\le x} \frac{1}{p} =& \log \log x + b+ O\Big(\frac{1}{\log x}\Big), \quad \mbox{and} \\
\sum_{p\le x} \frac {\log p}{p} =& \log x + O(1). 
\end{align*}
\end{lemma}

\begin{proof}
See, for example, parts (d) and (b) of \cite[Theorem 2.7]{Harper}.
\end{proof}

\subsection{Upper bounds for modular $L$-functions}
\label{sec2.5}

  Let $f \in H^*_{\kappa}(N)$. By a straightforward modification of the proof of \cite[Proposition 2.10]{G&Zhao24-12}, we arrive at the following upper bound for $\log |L(\sigma+it,\chi)|$.
\begin{lemma}
\label{lem: logLbound}
With the notation as above and assume the truth of GRH for $L(s,f)$, let $\sigma$, $t$ and $x \in \rear$ with $\sigma \geq 1/2$ and $x \geq 2$.  Let $\lambda_0=0.4912\ldots$ denote the unique positive real number satisfying $e^{-\lambda_0} = \lambda_0+ \lam^2_0/2$.  Then, for any $\lambda \geq \lam_0$,
\begin{align*}
\log |L(\sigma+it,f)| \le  \Re \sum_{\substack {p^l \leq x \\  l \geq 1}} \frac {(\alpha^{l}_p+\beta^l_p)}{lp^{l(\sigma+it+ \lambda/\log x)}} \frac{\log \leg {x}{p^l}}{\log x}
 + (1+\lam)(\log \sqrt{N}+\log (|t|+2))+ O\Big( \frac{\lambda}{\log x}+1\Big).
\end{align*}
\end{lemma}
  We apply the above lemma and argue as in the proof of \cite[Corollary 2.11]{G&Zhao24-12} to deduce the following bound for sums of $\log |L(1/2+it,f)|$ over various $t$.
\begin{lemma}
\label{lem: logLboundsimplified}
 With the notation as above and the truth of GRH for $L(s,f)$, let $k$ be a positive integer and $Q$, $a_1$, $a_2, \ldots, a_{k}$ be fixed positive real constants, $x\geq 2$.  Set $a:=a_1+\cdots+ a_{k}$.  Suppose $N$ is a integer,  $t_1,\ldots, t_{k}$ are fixed real numbers satisfying $|t_i|\leq N^A$. For any integer $n$, define
\[ h(n):=\frac{1}{2} \Big( \sum^{k}_{m=1}a_mn^{-it_m} \Big). \]
Then, for any $f \in H^*_{\kappa}(N)$ and $0 \leq \sigma -1/2 \ll 1/\log x$, we have
\begin{align}
\label{logLupperboundnonquad}
\begin{split}
\sum^{k}_{m=1} & a_m\log |L(\sigma+it_m,f)| \\
& \leq 2\Re \sum_{p\leq x} \frac{h(p)\lambda_f(p)}{p^{1/2+\max(\sigma-1/2, 1/\log x)}}\frac{\log x/p}{\log x}
    -\Re\sum_{p\leq \min (x^{1/2}, \log N)} \frac{h(p^2)(\lambda_f(p^2)-1)}{p}+(A+1)a\frac{\log N}{\log x}+O(1).
\end{split}
\end{align}
\end{lemma}

   We also note the following upper bounds on moments of modular $L$-functions.
\begin{lemma}
\label{prop: upperbound}
With the notation as above and the truth of GRH, let $k\geq 1$ be a fixed integer and ${\bf a}=(a_1,\ldots ,a_{k}),\ t=(t_1,\ldots ,t_{k})$
 be real $k$-tuples such that $a_i \geq 0$ and that $|t_i|\leq N^A$ for all $i$. Then for large prime number $N$ and $\sigma \geq 1/2$,
\begin{align*}
   \sum_{f \in H^*_{\kappa}(N)} \big| L\big(\sigma+it_1,f \big) \big|^{a_1} \cdots \big| L\big(\sigma+it_{k},f \big) \big|^{a_{k}} \ll_{{\bf a}} &  N(\log N)^{O(1)}.
\end{align*}
\end{lemma}
\begin{proof}
By H\"older’s inequality, it suffices to show that for any fixed integer $k \geq 1$ and any real  $|t|\leq N^A$, 
\begin{align*}
   \sum_{f \in H^*_{\kappa}(N)}\big| L\big(\sigma+it,f \big) \big|^{2k}  \ll_{{k}} &  N(\log N)^{O(1)}.
\end{align*}

   We deduce from Lemma \ref{RS}, Lemma \ref{lem: logLbound} and the bound $\lambda_f(p^2) \leq 3$ that for $x \leq N$, 
\begin{align}
\label{mainupper1}
\begin{split}
  \log |L(\sigma+it,f)| \leq  \Re \sum_{p\leq x} \frac{\lambda_f(p)}{p^{1/2+\max(\sigma-1/2, 1/\log x)+it}}\frac{\log x/p}{\log x}
    +2\log \log X+(A+1)\frac{\log X}{\log x}+O(1).
\end{split}
\end{align}

  Let $\mathcal{N}(V)$  be the number of $f \in H^*_{\kappa}(N)$ such that $\log|L(\sigma+it, f)|\geq V$. 
  Suppose that $f$ is counted by $\mathcal{N}(V)$ and also that $V \geq \max \{8\log \log X,10(A+1)\}$. We set $x=X^{10(A+1)/V}$ in \eqref{mainupper1}, so for such $f$, we have
\begin{align*}
\begin{split}
 \Re \sum_{p\leq x} \frac{\lambda_f(p)}{p^{1/2+\max(\sigma-1/2, 1/\log x)+it}}\frac{\log x/p}{\log x} \geq \frac V2.
\end{split}
\end{align*}

  We take $m = \lceil V/(100(A+1))\rceil $, where $\lceil x \rceil = \min \{ n \in \intz : n \geq x\}$, getting
\begin{align*} 
\begin{split}
  \Big( \frac {V}2 \Big)^{2m}  \mathcal{N}(V) \leq & \sum_{f \in H^*_{\kappa}(N)} \Big| \Re \sum_{p\leq x} \frac{\lambda_f(p)}{p^{1/2+\max(\sigma-1/2, 1/\log x)+it}}\frac{\log x/p}{\log x} \Big|^{2m} \ll N(2m-1)!!\Big (\sum_{p \leq x} \frac {1}{p}\Big )^{m},
\end{split}
\end{align*}
  where the last estimation above follows from the arguments given in the proof of \cite[Lemma 3.1]{HM07}, upon using Lemma \ref{lemma:delta} and Lemma \ref{lemma:deltasimplified}. \newline

  Now Stirling's formula (see \cite[(5.112)]{iwakow}) implies that
\begin{align*}
\begin{split}
 \Big( \frac me \Big)^m \leq m! \leq \sqrt{m} \Big( \frac {m }{e} \Big)^{m}.
\end{split}
\end{align*}
  It follows from this that
\begin{align*}
\begin{split}
 (2m-1)!!=\frac {(2m)!}{2^m m!}  \leq \sqrt{2m} \Big( \frac {2m }{e} \Big)^{m}.
\end{split}
\end{align*}  

  We then apply Lemma \ref{RS} to deduce from the above that for some constant $C_0$, 
\begin{align*} 
\begin{split}
 & \Big(\frac {V}2 \Big)^{2m}  \mathcal{N}(V)
 \ll  N  \sqrt{2m} \Big( \frac {2m }{e} \Big)^{m}\Big (\log \log N+C_0 \Big )^{m}.
\end{split}
\end{align*}
  
Thus we infer that if $V \geq e^{10000(A+1)(k+1)}(\log \log N+C_0 )$, then
\begin{equation}
\label{equ:bd-S-2}
\mathcal{N}(V)  \ll N\operatorname{exp}\left(-4kV\right).
\end{equation}

   We now set $V_0=\max \{ e^{10000(A+1)(k+1)}(\log \log N+C_0 ), 8\log \log X,10(A+1)\}$, so that the estimation given in \eqref{equ:bd-S-2} holds when $V \geq V_0$. Note moreover we have $\mathcal{N}(V) \ll N$ by \eqref{Hpmsize}. It then follows by partial summation that
\begin{align*}
   \sum_{f \in H^*_{\kappa}(N)}\big| L\big(\sigma+it,f\big) \big|^{2k} 
   \ll Ne^{2kV_0}+\sum^{\infty}_{V=V_0}\mathcal{N}(V)e^{(V+1)2k} \ll  N(\log N)^{O(1)}.
\end{align*}  
  This completes the proof of the lemma.
\end{proof}

\section{Proof of Theorem \ref{t1} }
\label{sec:upper bd}

Setting $\sigma=1/2$ in \eqref{logLupperboundnonquad} and exponentiating both sides, we get
\begin{align}
\label{basicest}
\begin{split}
 \big|  L & \big( \tfrac12+it_1,f\big) \big|^{a_1} \cdots \big| L\big(\tfrac12+it_{k},f \big) \big|^{a_{k}} \\
 & \ll   \exp \left( \Re\sum_{\substack{  p \leq x }} \frac{2h(p)\lambda_f(p)}{p^{1/2+1/\log x}}
 \frac{\log (x/p)}{\log x} -
 \Re \sum_{\substack{  p \leq  \min (x^{1/2}, \log X) }} \frac{h(p^2)(\lambda_f (p^2)-1)}{p} 
 +(A+1)a\frac{\log N}{\log x}\right) .
\end{split}
 \end{align}
 
Follwoing the approach of A. J. Harper in \cite{Harper}, we define for a large number $N$,
$$ \alpha_{0} = \frac{\log 2}{\log N}, \;\;\;\;\; \alpha_{i} = \frac{20^{i-1}}{(\log\log N)^{2}} \;\;\; \mbox{for all} \; i \geq 1, \quad
\mathcal{J} = \mathcal{J}_{k,N} = 1 + \max\{i : \alpha_{i} \leq 10^{-T} \} . $$

   We set
\[ {\mathcal M}_{i,j}(f) = \sum_{N^{\alpha_{i-1}} < p \leq N^{\alpha_{i}}}  \frac{2h(p)\lambda_f (p)}{p^{1/2+1/(\log N^{\alpha_{j}})}} \frac{\log (N^{\alpha_{j}}/p)}{\log N^{\alpha_{j}}}, \quad 1\leq i \leq j \leq \mathcal{J} , \]
and
\[ P_{m}(f )= -\sum_{2^{m} < p \leq 2^{m+1}} \frac{h(p^2)(\lambda_f (p^2)-1)}{p}, \quad  0 \leq m \leq \frac{\log\log N}{\log 2}. \]

 We also define for $0 \leq j \leq \mathcal{J}$,
\begin{align*}
 \mathcal{S}(j) =& \left\{ f \in H^*_{\kappa}(N) : | {\mathcal M}_{i,l}(f)| \leq \alpha_{i}^{-3/4} \; \; \mbox{for all}  \; 1 \leq i \leq j, \; \mbox{and} \; i \leq l \leq \mathcal{J}, \right. \\
 & \left. \hspace*{2cm} \text{but }  | {\mathcal M}_{j+1,l}(f)| > \alpha_{j+1}^{-3/4} \; \text{ for some } j+1 \leq l \leq \mathcal{J} \right\} . \\
 \mathcal{S}(\mathcal{J}) =& \left\{ f \in H^*_{\kappa}(N) : |{\mathcal M}_{i, \mathcal{J}}(f)| \leq \alpha_{i}^{-3/4} \; \mbox{for all}  \; 1 \leq i \leq \mathcal{J} \right\}, \; \mbox{and} \\
\mathcal{P}(m) =&  \left\{ f \in H^*_{\kappa}(N) : | P_{m}(f)| > 2^{-m/10} , \; \text{but} \; | P_{n}(f)| \leq 2^{-n/10} \; \mbox{for all} \; m+1 \leq n \leq \frac{\log\log N}{\log 2} \right\}.
\end{align*}

    We shall set $x=N^{\alpha_j}$ for $j \geq 1$ in \eqref{basicest} in what follows.  So we may assume that the second sum on the right-hand side of \eqref{basicest} is over $p \leq \log N$.  Then $|P_{n}(f)| \leq 2^{-n/10}$ for all $n$ if $f\not \in \mathcal{P}(m)$ for any $m$, which implies that
\[ \Re \sum_{\substack{  p \leq \log N }}  \frac{h(p^2)(\lambda_f (p^2)-1)}{p} 
= O(1). \]
 As the treatment for case $f \not \in \mathcal{P}(m)$ for any $m$ is easier compared to the other cases, we may assume that $f \in \mathcal{P}(m)$ for some $m$.   We further note that
$$ \mathcal{P}(m)= \bigcup_{m=0}^{\log \log N/2}\bigcup_{j=0}^{ \mathcal{J}} \Big (\mathcal{S}(j)\bigcap \mathcal{P}(m) \Big ), $$
so that it suffices to show that
\begin{align}
\label{sumovermj}
\begin{split}
  & \sum_{m=0}^{\log \log N/2}\sum_{j=0}^{\mathcal{J}}\sum_{\chi \in \mathcal{S}(j)\bigcap \mathcal{P}(m)} \big| L\big(1/2+it_1,f\big) \big|^{a_1} \cdots \big| L\big(1/2+it_{k},f \big) \big|^{a_{k}} \ll  N.
\end{split}
\end{align}

   Observe that
\begin{align*}
\text{meas}(\mathcal{P}(m)) \leq \sum_{f \in H^*_{\kappa}(N)}
\Big (2^{m/10} |P_m(f)| \Big )^{2\lceil 2^{m/2}\rceil }.
\end{align*}
Following the proof of Lemma \ref{prop: upperbound} together with the estimation that $|h(p^2)(\lambda_f(p^2)-1)| \leq 2a$, we see that for $m$ large enough in terms of $a$,
\begin{align}
\label{Pmest}
\begin{split}
  \text{meas}(\mathcal{P}(m)) \ll & N(2^{m/10})^{2\lceil 2^{m/2}\rceil } \Big (\sum_{2^m < p } \frac {4a^2}{p^2}\Big )^{\lceil 2^{m/2}\rceil }
    \ll N 2^m (4a^22^{-m+m/5})^{2^{m/2}} \ll N2^{-2^{m/2}}.
\end{split}
 \end{align}
Then the Cauchy-Schwarz inequality, Lemma \ref{prop: upperbound} and \eqref{Pmest} yield that if $2^{m} \geq (\log\log N)^{3}$ and $N$ large enough,
\begin{align*}
 \sum_{f \in  \mathcal{P}(m)}  \big| L\big( \tfrac12+it_1,f\big) \big|^{a_1} \cdots \big| L\big(\tfrac12+it_{k},f \big) \big|^{a_{k}} 
\leq & \left( \text{meas}(\mathcal{P}(m)) \cdot
\sum_{f \in H^*_{\kappa}(N)}\big| L\big(\tfrac12+it_1,f \big) \big|^{2a_1} \cdots \big| L\big(\tfrac12+it_{k},f \big) \big|^{2a_{k}} \right)^{1/2}
 \\
 \ll & \left( N \exp\left( -(\log 2)(\log\log N)^{3/2} \right) N (\log N)^{O(1)} \right)^{1/2} \ll N.
\end{align*}

   The above now allows us to focus on the case $0 \leq m \leq (3/\log 2)\log\log\log N$.  Similarly,
\begin{align}
\label{S0est}
\begin{split}
\text{meas}(\mathcal{S}(0)) \ll & \sum_{f \in H^*_{\kappa}(N)}
\sum^{\mathcal{J}}_{l=1}
\Big ( \alpha^{3/4}_{1}{|\mathcal
M}_{1, l}(f)| \Big)^{2\lceil 1/(10\alpha_{1})\rceil }=
\sum^{\mathcal{J}}_{l=1}\sum_{f \in H^*_{\kappa}(N)}\Big ( \alpha^{3/4}_{1}{|\mathcal
M}_{1, l}(f)| \Big)^{2\lceil 1/(10\alpha_{1})\rceil } .
\end{split}
\end{align}
   Note that
\begin{align*}
 \mathcal{J} \leq \log\log\log N , \; \alpha_{1} = \frac{1}{(\log\log N)^{2}} , \; \mbox{and} \; \sum_{p \leq N^{1/(\log\log N)^{2}}} \frac{1}{p} \leq \log\log N ,
\end{align*}
  where the last bound follows from Lemma \ref{RS}. We apply these estimates and argue as in the proof of Lemma \ref{prop: upperbound} to evaluate the last sums in \eqref{S0est} above in a manner similar to the computation in \eqref{Pmest}.  This renders
\begin{align*}
\text{meas}(\mathcal{S}(0)) \ll &
\mathcal{J}N e^{-1/\alpha_{1}}\ll N e^{-(\log\log N)^{2}/10}  .
\end{align*}
Using the Cauchy-Schwarz inequality, Proposition \ref{prop: upperbound} and the above, we arrive at
\begin{align*}
\sum_{f \in  \mathcal{S}(0)}  \big| L\big(\tfrac12+it_1,f\big) \big|^{a_1} \cdots \big| L\big(\tfrac12+it_{k},f \big) \big|^{a_{k}}  
\leq &   \left( \text{meas}(\mathcal{S}(0)) \cdot
\sum_{f \in H^*_{\kappa}(N)}\big| L\big(\tfrac12+it_1,f\big) \big|^{2a_1} \cdots \big| L\big(\tfrac12+it_{k},f \big) \big|^{2a_{k}} \right)^{1/2}
 \\
 \ll & \left( N \exp\left( -(\log\log N)^{2}/10 \right) N (\log N)^{O(1)} \right)^{1/2} \ll N.
\end{align*}

  Thus we may further assume that $j \geq 1$. Note that when $f \in \mathcal{S}(j)$, we set $x=N^{\alpha_j}$ in \eqref{basicest} to arrive at
\begin{align*}
\begin{split}
 & \big| L\big(1/2+it_1,f\big) \big|^{a_1} \cdots \big| L\big(1/2+it_{k},f \big) \big|^{a_{k}} \ll \exp \left(\frac {(A+1)a}{\alpha_j} \right) \exp \Big (
 \Re\sum^j_{i=1}{\mathcal M}_{i,j}(f)+ \Re\sum^{\log \log N/2}_{l=0}P_l(f) \Big ).
\end{split}
 \end{align*}

   When restricting the sum of $\big| L\big(1/2+it_1,f \big) \big|^{a_1} \cdots \big| L\big(1/2+it_{k},f \big) \big|^{a_{k}}$ over $\mathcal{S}(j)\bigcap \mathcal{P}(m)$, our treatments below require us to separate the sums over $p \leq 2^{m+1}$ on the right-hand side of the above expression from those over $p>2^{m+1}$. For this, we note that if $f \in \mathcal{P}(m)$, then
\begin{align*}
\begin{split}
  \Re \sum_{  p \leq 2^{m+1}} & \frac{2h(p)\lambda_f (p)}{p^{1/2+1/(\log N^{\alpha_{j}})}}  \frac{\log (N^{\alpha_{j}}/p)}{\log N^{\alpha_{j}}}-
  \Re \sum_{p \leq \log N} \frac{h(p^2)(\lambda_f (p^2)-1)}{p} 
   \\
 \leq &\Re \sum_{  p \leq 2^{m+1}}  \frac{2h(p)\lambda_f (p)}{p^{1/2+1/(\log N^{\alpha_{j}})}}  \frac{\log (N^{\alpha_{j}}/p)}{\log N^{\alpha_{j}}}-
 \Re \sum_{p \leq 2^{m+1}} \frac{h(p^2)(\lambda_f (p^2)-1)}{p}  +O(1)  \leq a2^{m/2+3}+O(1).
\end{split}
 \end{align*}

 It follows that
\begin{align}
\label{LboundinSP}
\begin{split}
 \sum_{f \in \mathcal{S}(j)\bigcap \mathcal{P}(m)} & \big| L\big(\tfrac12+it_1, f\big) \big|^{a_1} \cdots \big| L\big(\tfrac12+it_{k}, f \big) \big|^{a_{k}}   \\
   \ll & e^{a2^{m/2+4}} \sum_{\chi \in \mathcal{S}(j)\bigcap \mathcal{P}(m)}
 \exp \left(\frac {(A+1)a}{\alpha_j} \right)\exp \Big (  \Re \sum_{ 2^{m+1}< p \leq N^{\alpha_j} }
 \frac{2h(p)\lambda_f (p)}{p^{1/2+1/(\log N^{\alpha_{j}})}}  \frac{\log (N^{\alpha_{j}}/p)}{\log N^{\alpha_{j}}}\Big )
  \\
\ll &  e^{a2^{m/2+4}} \exp \left(\frac {(A+1)a}{\alpha_j} \right)\sum_{f \in \mathcal{S}(j)} \Big (2^{m/10}|P_m(f)| \Big )^{2\lceil 2^{m/2}\rceil }
\exp \Big ( \Re{\mathcal M}'_{1,j}(f)+\Re \sum^j_{i=2}{\mathcal M}_{i,j}(f)\Big ),
\end{split}
 \end{align}
   where
\begin{align*}
\begin{split}
 {\mathcal M}'_{1,j}(f):= \sum_{ 2^{m+1}< p \leq N^{\alpha_1} }
 \frac{2h(p)\lambda_f (p)}{p^{1/2+1/\log N^{\alpha_{j}}}}  \frac{\log (N^{\alpha_{j}}/p)}{\log N^{\alpha_{j}}}.
\end{split}
 \end{align*}

    We note that if $0 \leq m \leq (3/\log 2)\log\log\log N$ and $N$ large enough, then
\begin{align*}
\begin{split}
  \sum_{ p< 2^{m+1}  }
 \frac{2h(p)\lambda_f (p)}{p^{1/2+1/\log N^{\alpha_{j}}}}  \frac{\log (N^{\alpha_{j}}/p)}{\log N^{\alpha_{j}}} \leq 2a\sum_{ p< 2^{m+1}  }
 \frac{1}{\sqrt{p}} \leq \frac {100a \cdot 2^{m/2}}{m+1} \leq 100a(\log \log N)^{3/2}(\log \log \log N)^{-1},
\end{split}
 \end{align*}
   where the last estimate above follows from partial summation and Lemma \ref{RS}. \newline

We infer from this that if $\chi \in \mathcal{S}(j)$ and $N$ is large enough, then
\begin{align*}
\begin{split}
\left| {\mathcal M}'_{1,j}(f) \right| \leq 100a(\log \log N)^{3/2}(\log \log \log N)^{-1}+ \left| {\mathcal M}_{1,j}(f) \right| \leq 1.01\alpha^{-3/4}_1=1.01(\log \log N)^{3/2}.
\end{split}
 \end{align*}

Thus we also have $|\frac 12 {\mathcal M}'_{1,j}(f)|$, $|\frac 12 \overline{{\mathcal M}'_{1,j}(f)}| \leq 1.01(\log \log N)^{3/2}$. We then apply \cite[Lemma 5.2]{Kirila} to see that
\begin{align*}
\begin{split}
\exp \Big ( \Re {\mathcal M}'_{1,j}(f)\Big )=& \exp \Big ( \tfrac 12 {\mathcal M}'_{1,j}(f)+\tfrac 12 \overline{{\mathcal M}'_{1,j}(f)} \Big ) \ll  E_{e^2a\alpha^{-3/4}_1}(\frac 12{\mathcal M}'_{1,j}(f))E_{e^2a\alpha^{-3/4}_1}(\frac 12\overline{{\mathcal M}'_{1,j}(f)}) \\
=& \Big | E_{e^2a\alpha^{-3/4}_1}(\frac 12{\mathcal M}'_{1,j}(f)) \Big |^2, 
\end{split}
 \end{align*}
    where for any real numbers $x$ and $\ell \geq 0$, we define
\begin{align*}
  E_{\ell}(x) = \sum_{j=0}^{\lceil\ell \rceil} \frac {x^{j}}{j!}.
\end{align*}

   We then estimate $\exp \Big ( \Re {\mathcal M}'_{1,j}(f)+ \Re\sum^j_{i=2}{\mathcal M}_{i,j}(f)\Big )$ similarly by noting that $|{\mathcal M}_{i, j}(f)|\leq  \alpha^{-3/4}_i$ for $f \in \mathcal{S}(j)$. This leads to 
\begin{align*}
\begin{split}
\exp \Big ( \Re {\mathcal M}'_{1,j}(f)+ \Re\sum^j_{i=2}{\mathcal M}_{i,j}(f)\Big ) \ll \Big | E_{e^2a\alpha^{-3/4}_1}(\tfrac 12{\mathcal M}'_{1,j}(f)) \Big |^2
\prod^j_{i=2}\Big | E_{e^2a\alpha^{-3/4}_i}(\tfrac 12{\mathcal M}_{i,j}(f)) \Big |^2.
\end{split}
 \end{align*}

 It thus follows from the description on $\mathcal{S}(j)$ that when $j \geq 1$,
\begin{align*}
\begin{split}
\sum_{f \in \mathcal{S}(j)\bigcap \mathcal{P}(m)} & \big| L\big(1/2+it_1,f\big) \big|^{a_1} \cdots \big| L\big(1/2+it_{k},f \big) \big|^{a_{k}}  \\
 \ll & e^{a2^{m/2+4}} \exp \left(\frac {(A+1)a}{\alpha_j} \right)
 \sum^{R}_{l=j+1} \sum_{f \in H^*_{\kappa}(N)}\Big (2^{m/10}|P_m(f)| \Big )
 ^{2\lceil 2^{m/2}\rceil } \\
& \hspace*{2cm} \times \exp \Big ( \Re {\mathcal M}'_{1,j}(f)+ \Re \sum^j_{i=2}{\mathcal M}_{i,j}(f)\Big )\Big ( \alpha^{3/4}_{j+1}\big|{\mathcal
M}_{j+1, l}(f)\big |\Big)^{2\lceil 1/(10\alpha_{j+1})\rceil } \\
\ll &  e^{a2^{m/2+4}}\exp \left(\frac {(A+1)a}{\alpha_j} \right)\sum^{R}_{l=j+1}
\sum_{f \in H^*_{\kappa}(N)}
\Big (2^{m/10}|P_m(f)| \Big )^{2\lceil 2^{m/2}\rceil } \\
& \hspace*{2cm} \times \Big | E_{e^2a\alpha^{-3/4}_1}(\tfrac 12{\mathcal M}'_{1,j}(f)) \Big |^2
\prod^j_{i=2}\Big | E_{e^2a\alpha^{-3/4}_i}(\tfrac 12{\mathcal M}_{i,j}(f)) \Big |^2\Big ( \alpha^{3/4}_{j+1}\big |{\mathcal
M}_{j+1, l}(f)\big |\Big)^{2\lceil 1/(10\alpha_{j+1})\rceil } .
\end{split}
 \end{align*}

 Note that we have for $1 \leq j \leq \mathcal{I}-1$,
\begin{align*}
\mathcal{I}-j \leq \frac{\log(1/\alpha_{j})}{\log 20}  \quad \mbox{and} \quad \sum_{N^{\alpha_{j}} < p \leq N^{\alpha_{j+1}}} \frac{1}{p}
 = \log \alpha_{j+1} - \log \alpha_{j} + o(1) = \log 20 + o(1) \leq 10 .
\end{align*}

  We argue again as in the proof of Lemma \ref{prop: upperbound} to see that by taking $N$ large enough,
\begin{align*}
\begin{split}
    \sum^{ \mathcal{I}}_{l=j+1}
\sum_{f \in H^*_{\kappa}(N)} &
\Big (2^{m/10}|P_m(f)| \Big )^{2\lceil 2^{m/2}\rceil } \\
& \hspace*{1cm} \times \Big | E_{e^2a\alpha^{-3/4}_1}(\tfrac 12{\mathcal M}'_{1,j}(f)) \Big |^2
\prod^j_{i=2}\Big | E_{e^2a\alpha^{-3/4}_i}(\tfrac 12{\mathcal M}_{i,j}(f)) \Big |^2\Big ( \alpha^{3/4}_{j+1}\big |{\mathcal
M}_{j+1, l}(f)\big |\Big)^{2\lceil 1/(10\alpha_{j+1})\rceil }  \\
\ll & N(\mathcal{I}-j)e^{-44a(A+1)/\alpha_{j+1}} 2^m (2^{-2m/15})^{\lceil 2^{m/2}\rceil } \prod_{p \leq N^{\alpha_j}}\left( 1+O \left( \frac 1{p^2} \right) \right) \\
\ll & N(\mathcal{I}-j)e^{-44a(A+1)/\alpha_{j+1}} 2^m (2^{-2m/15})^{\lceil 2^{m/2}\rceil } \prod_{p \leq N}\left( 1+O \left( \frac 1{p^2} \right) \right) \\
\ll & e^{-42a(A+1)/\alpha_{j+1}} 2^m (2^{-2m/15})^{\lceil 2^{m/2}\rceil }. 
\end{split}
 \end{align*}
 
   We then conclude from the above and \eqref{LboundinSP} that (by noting that $20/\alpha_{j+1}=1/\alpha_j$)
\begin{align*}
\begin{split}
 & \sum_{\chi \in \mathcal{S}(j)\bigcap \mathcal{P}(m)} \big| L\big(\tfrac12+it_1,f\big) \big|^{a_1} \cdots \big| L\big(\tfrac12+it_{k},f \big) \big|^{a_{k}} \ll  e^{-a(A+1)/(10\alpha_{j})}2^m  e^{a2^{m/2+4}} (2^{-2m/15})^{\lceil 2^{m/2}\rceil }N.
\end{split}
 \end{align*}

As the sum of the right side expression over $m$ and $j$ converges, the above implies \eqref{sumovermj} and this completes the proof of Theorem \ref{t1}.
  
\vspace*{.5cm}

\noindent{\bf Acknowledgments.}  P. G. is supported in part by NSFC grant 12471003 and L. Z. by the FRG grant PS71536 at the University of New South Wales. 
\bibliography{biblio}
\bibliographystyle{amsxport}

\end{document}